\documentclass{article}
\usepackage[utf8]{inputenc}
\usepackage{mathtools,amsfonts,amssymb,amsthm,tikz-cd,tikz,amsmath}
\newtheorem{theorem}{Theorem}
\newtheorem{lemma}[theorem]{Lemma}

\newtheorem{remark}{Remark}
\newtheorem{example}{Example}
\newtheorem{proposition}{Proposition}
\theoremstyle{definition}
\newtheorem{definition}{Definition}
\newtheorem{problem}{Problem}
\title{Information geometry and Frobenius algebra}
\author{Ruichao Jiang, Javad Tavakoli, Yiqiang Zhao}
\date{Oct 2020}

\begin{document}

\maketitle
\begin{abstract}
 We show that a Frobenius sturcture is equivalent to a dually flat sturcture in information geometry. We define a multiplication structure on the tangent spaces of statistical manifolds, which we call the statistical product. We also define a scalar quantity, which we call the Yukawa term. By showing two examples from statistical mechanics, first the classical ideal gas, second the quantum bosonic ideal gas, we argue that the Yukawa term quantifies information generation, which resembles how mass is generated via the 3-points interaction of two fermions and a Higgs boson (Higgs mechanism). In the classical case, The Yukawa term is identically zero, whereas in the quantum case, the Yukawa term diverges as the fugacity goes to zero, which indicates the Bose-Einstein condensation.
\end{abstract}
\section{Introduction}
Amari in his own monograph \cite{amari} writes that major journals rejected his idea of dual affine connections on the grounds of their nonexistence in the mainstream Riemannian geometry. Later on, this dualistic structure turned out to be useful in information science, e.g. EM algorithms. The aim of this paper is to show the Frobenius algebra structure in information geometry and thus give a physical meaning for the Amari-Chentsov tensor.

In information geometry, instead of the usual Levi-Civita connection, a nonmetrical connections $\nabla$ and its dual $\nabla^*$ are used. Associated with these two connections is a mysterious tensor, called the Amari-Chentsov tensor (AC tensor) and a one-parameter family of flat connections, called the $\alpha$-connections. Neither the AC tensor nor the $\alpha$-connections are present in Riemannian geometry.

Frobenius algebra is an algebraic structure discovered independently by Saito and by Dubrovin. It turns out to be an extremely general structure appearing in seemingly unrelated fields of mathematics and physics, e.g.
\begin{itemize}
    \item Physics: Topological field theory \cite{dubrovin} and conformal field theory \cite{fuchs};
    \item Algebraic geometry: Mirror symmetry and quantum cohomology \cite{ruan}, and unfolding of singularities of complex hypersurfaces \cite{saito};
    \item Integrable systems: Painlev\'e IV \cite{dubrovin2} and Hamiltonian systems \cite{dubrovin3}
\end{itemize}
At page 311 of \cite{dubrovin2}, Dubrovin himself has remarked on the connection between mathematical statistics and Frobenius algebra. In this paper, we will show that the dually-flat structure and $\alpha$-structure in information geometry are equivalent to a Frobenius algebra structure on the tangent bundle. The Frobenius algebra gives rise to a product on every tangent plane. This product structure completely characterizes the AC tensor. We also extract a scalar quantity out of the AC tensor. We call this scalar the Yukawa information and we give two examples.

The organization of this paper is as follows: Section 2 is a short introduction to information geometry; Section 3 is a short introduction to Frobenius algebra; Section 4 shows the equivalence of the Frobenius algebra sturcture and dually flat structure; in Section 5 we define the statistical product and the Yukawa term and we give two examples; Section 6 is the discussion and we propose a further problem.
\section{Information geometry}
In this section, a concise description of information geometry is given. The exposition here is not complete but is geared to show the relationship between information geometry and Frobenius algebra. For a complete introduction, see the monograph \cite{amari}.
\begin{definition}[Statistical Manifold]
A statistical manifold is a differential manifold $M$ equipped with a Riemannian metric $g_{ij}$ and a totally-symmetric 3rd-order tensor $C_{ijk}$, called the Amari-Chentsov tensor.
\end{definition}
Alternatively, one can define a statistical manifold to be 
\begin{definition}\label{dualconnection}
A statistical manifold is Riemannian manifold $(M,g)$ with a torsion-free affine connection $\nabla$ and a dual connection defined by
\begin{equation*}
    \nabla^*=2^{LC}\nabla-\nabla,
\end{equation*}
where $^{LC}\nabla$ is the Levi-Civita connection.
\end{definition}
The following result shows that these two definitions are equivalent
\begin{proposition}
\begin{align*}
    &C_{ijk}=\Gamma^*_{ijk}-\Gamma_{ijk},\\
    &^{LC}\Gamma_{ijk}=\frac{1}{2}\left(\Gamma_{ijk}+\Gamma^*_{ijk}\right).
\end{align*}
\end{proposition}
\begin{proof}
See page 133 in \cite{amari}.
\end{proof}
Since the connection $\nabla$ is usually not the Levi-Civita connection, the metric tensor $g$ is not preserved under the parallel transport $\Pi$ induced by $\nabla$. However, $g$ is dually preserved by $\nabla$ and its dual $\nabla^*$, i.e.
\begin{proposition}
Let $\Pi$ and $\Pi^*$ be the parallel transport induced by $\nabla$ and its dual $\nabla^*$ along some curve $\gamma$ from $x$ to $y\in M$. Then
\begin{equation*}
    g_x(X,Y)=g_y(\Pi^*X,\Pi Y),
\end{equation*}
where $X,Y\in T_xM$.
\end{proposition}
\begin{remark}
In \cite{amari}, this proposition is taken as the definition of the dual connection and the Definition \ref{dualconnection} is derived from it. In information geometry, usually the metric $g$ is not preserved by the connection $\nabla$ i.e.
\begin{equation*}
    g(X,Y)\neq g(\Pi X,\Pi Y).
\end{equation*}
Therefore, the motivation of the dual connection $\nabla^*$ is that working together, the connection and its dual preserve the metric
\begin{equation*}
    g(X,Y)=g(\Pi^* X,\Pi Y)
\end{equation*}
\end{remark}
If in addition $\nabla$ is a flat connection, i.e. the Riemann curvature tensor vanishes identically, then we have a dually flat structure.
\begin{proposition}[Dually flat]
If $\nabla$ is flat, then its dual $\nabla^*$ is also flat. In this case, the manifold $M$ is said to be dually flat.
\end{proposition}
\begin{proof}
See page 137 in \cite{amari}.
\end{proof}
The final ingredient from information geometry is $\alpha$-connections.
\begin{definition}[$\alpha$-connections]
Let $(M,g,C)$ be a statistical manifold. The $\alpha$-connections are one-parameter connections
\begin{equation*}
    \Gamma^\alpha_{ijk}=^{LC}\Gamma_{ijk}+\alpha C_{ijk}.
\end{equation*}
\end{definition}
\begin{remark}
The definition in \cite{amari} is $\Gamma^\alpha_{ijk}=^{LC}\Gamma_{ijk}+\frac{\alpha}{2}C_{ijk}$. However, the factor $\frac{1}{2}$ does not play an important role.
\end{remark}
The following proposition collects some properties of the $\alpha$-connections
\section{Frobenius algebra}
This section serves as an introduction to Frobenius algebra and Frobenius manifold.
\begin{definition}[Frobenius algebra]
A Frobenius algebra is a triple $(A,g,\circ)$ such that the following holds:
\begin{enumerate}
    \item $(A,\circ)$ is a unital commutative associative algebra over a filed $K$;
    \item $g$ is a nondegenerate symmetric bilinear map such that
    $$
    g(X\circ{Y},Z)=g(X,Y\circ{Z}),
    $$
    where $X,Y,Z\in{A}$.
\end{enumerate}
\end{definition}
As to any algebra, it is useful to introduce structural constants to represent the product operation on a Frobenius algebra.
\begin{definition}[Structural constants]
Let $\{e_1,...,e_n\}$ be a basis of $A$ as a vector space over $K$. The structural constants for $A$ in this basis is defined via
$$
e_i\circ{e_j}=C^k_{ij}e_k.
$$
\end{definition}
\begin{remark}
These structural constants also appear in conformal field theory and is known as operator product expansion (OPE).
\end{remark}
Using the structural constants, we can explicitly write out the conditions on being a Frobenius algebra.
\begin{lemma}
Let $(A,g,\circ)$ be an algebra. Then
\begin{enumerate}
    \item $\circ$ is commutative if and only if
    \begin{equation}\label{commutative}
        C^k_{ij}=C^k_{ji};
    \end{equation}
    \item $\circ$ is associative if and only if
    \begin{equation}\label{associative}
        C_{ij}^lC_{lk}^m=C_{jk}^lC_{li}^m;
    \end{equation}
    \item $g(X\circ{Y},Z)=g(X,Y\circ{Z})$ if and only if
    \begin{equation}\label{invariance}
        C_{ijk}\coloneqq g_{lk}C^l_{ij}=C_{jki}.
    \end{equation}
\end{enumerate}
\end{lemma}
\begin{remark}
By Eqn.(\ref{commutative}) and Eqn.(\ref{invariance}), $C_{ijk}$ is totally symmetric in all three indices. The Eqn.(\ref{associative}) is known in conformal field theory as the conformal bootstrap.
\end{remark}
With Frobenius algebra, we can define Frobenius manifold.
\begin{definition}[Frobenius manifold]
A Frobenius manifold is a smooth manifold $M$ such that the following holds:
\begin{enumerate}
    \item Each tangent space $T_xM$ is equipped with a Frobenius algebra $(T_xM,\circ_{x},g_x)$;
    \item The Levi-Civita connection $\nabla^{LC}$ is flat;
    \item The tensor $C(X,Y,Z)\coloneqq g_x(X\circ_xY,Z)$ is totally symmetric, where $X,Y,Z\in{T_xM}.$
\end{enumerate}
\end{definition}
The flatness of the Levi-Civita connection implies the coincidence of the covariant derivative and the usual partial derivative. Then Poincar\'e lemma implies that there exists locally a function $\Psi$ on $M$ such that
\begin{equation}
    C_{ijk}=\frac{\partial^3\Psi}{\partial{x^i}\partial{x^j}\partial{x^k}}.
\end{equation}
\begin{remark}
The function $\Psi$ is known as the Gromov-Witten potential in mathematical physics. The tensor $C_{ijk}$ obtained is known as the three-points Gromov-Witten invariants, correlation function, or the Yukawa coupling in the A-model of the topological field theory.
\end{remark}
With Gromov-Witten potential, we can introduce WDVV equation.
\begin{definition}[WDVV]
The WDVV equation for the Gromov-Witten potential $\Psi$ is
$$
\frac{\partial^3\Psi}{\partial x^i \partial x^j \partial x^a}g^{ab}\frac{\partial^3\Psi}{\partial x^b \partial x^k \partial x^l}=\frac{\partial^3\Psi}{\partial x^j \partial x^k \partial x^a}g^{ab}\frac{\partial^3\Psi}{\partial x^b \partial x^i \partial x^l}.
$$
\end{definition}
Finally, we introduce the last way to describe the Frobenius algebra structure on a manifold.
\begin{definition}[Dubrovin connections]
Let $\alpha\in\mathbb{R}$ be a parameter. The Dubrovin connections are a one-parameter family of connections $^D\nabla^\alpha$ defined by
$$
^{D}\nabla^{\alpha}_{X}(Y)\coloneqq^{LC}\nabla_{X}(Y)+\alpha X\circ Y.
$$
\end{definition}
\begin{remark}
The parameter $\alpha$ originates from the theory of integrable system and is called the spectral parameter there.
\end{remark}
The following proposition connects three characterizations of the Frobenius algebra.
\begin{proposition}\label{equivelence}
The following are equivalent
\begin{enumerate}
    \item $(A,g,\circ)$ is a Frobenius algebra;
    \item There exists Gromov-Witten potential $\Psi$ satisfying the WDVV equation;
    \item The Dubrovin connections $^D\nabla^{\alpha}$ are flat for any $\alpha\in\mathbb{R}$.
\end{enumerate}
\end{proposition}
\begin{remark}
This proposition shows many facets of the Frobenius algebra. The first term is purely algebraic. The second term is from integrable systems. And the third term is differential-geometric, which shows a general phenomena that an integrable PDE also has a reformulation as the vanish of some curvature.
\end{remark}
Besides the usual covariant connection, we also need the notion of contravariant connection.
\begin{definition}[Contravariant connection]
Let $(M,g)$ be a Riemannian manifold with an affine connection $\nabla$. The contravariant connection with respect to $\nabla$ is defined by
\begin{align*}
    \nabla^*:T^*M\otimes T^*M&\longrightarrow T^*M,\\
    (\alpha,\beta)&\longmapsto\nabla_{\alpha^\sharp}\beta,
\end{align*}
where $\sharp$ is the canonical isomorphism between the cotangent space and the tangent space.
The Christoffel symbol of $\nabla^*$ is
\begin{equation*}
    \Gamma^{ij}_k=-g^{is}\Gamma^j_{sk}.
\end{equation*}
\end{definition}
\begin{remark}
The minus sign in the Christoffel symbol is a consequence of the formula for the covariant derivative of a 1-form.
\end{remark}
\section{Frobenius algebra and information geometry}
After the introduction of the information geometry and Frobenius algebra, it is already transparent that the dually flat structure is equivalent to a Frobenius algebra structure
\begin{proposition}
To a dually flat sturcture $(M,g,\nabla,\nabla^*)$ there associates with a Frobenius algebra structure.
\end{proposition}
\begin{proof}
To define a product on a tangent space, it suffices to prescribe its structural constants in a given basis $\{\partial_i\}$. The structural constants are given up a scalar multiple by
\begin{equation}
    C^i_{jk}=\Gamma^i_{ik}-\Gamma^{*i}_{jk},
\end{equation}
where $\Gamma^i_{ik}$ and $\Gamma^{*i}_{jk}$ are Christoffel symbols of $\nabla$ and $\nabla^*$.
\end{proof}
Conversely, to a Frobenius algebra sturcture, the Dubrovin connections give two dually flat connections. The only thing that needs to be proven is that the two connections given by the Dubrovin connections preserve the metric tensor $g$. The following proposition shows that is it the case.
\begin{proposition}
The canonical pairing $<,>:T_x^*M\times T_xM\rightarrow\mathbb{R}$ is preserved along the parallel transport given by the Dubrovin connection $^D\nabla^\alpha$. 
\end{proposition}
\begin{proof}
Let $\{\frac{\partial}{\partial x^i}\}$ denote a basis of the tangent plane $T_xM$ at some point $x\in M$ and $\{\textrm{d}x^i\}$ the corresponding dual basis of the cotangent plane $T_x^*M$. Then we need to prove that $\forall$i, j, and k
    \begin{equation*}
        \frac{\partial}{\partial x^i}\left<\textrm{d}x^j,\frac{\partial}{\partial x^k}\right>=0.
    \end{equation*}
This is shown by the direct computation
\begin{align*}
    &\frac{\partial}{\partial x^i}\left<\textrm{d}x^j,\frac{\partial}{\partial x^k}\right>=\left<g_{is}\nabla^{* \textrm{d}x^s}\textrm{d}x^j,\frac{\partial}{\partial x^k}\right>+\left<\textrm{d}x^j,\nabla_{\frac{\partial}{\partial x^i}}\frac{\partial}{\partial x^k}\right>\\
    &=\left<-g_{is}g^{st}\Gamma^j_{tr}\textrm{d}x^r-\alpha g_{is}g^{st}C^j_{tr}\textrm{d}x^r,\frac{\partial}{\partial x^k}\right>+\left<\textrm{d}x^j,\Gamma^r_{ik}\frac{\partial}{\partial x^r}+\alpha C^r_{ik}\frac{\partial}{\partial x^r}\right>\\
    &=-\delta^t_i\Gamma^j_{tk}-\alpha\delta^t_i C^j_{tk}+\Gamma^j_{ik}+\alpha C^j_{ik}\\
    &=0.
\end{align*}
\end{proof}
We have shown the equivalence of the Frobenius algebra and the dually flat structure in information geometry. What if the statistical manifold is not dually flat? It turns out that we may still define the statistical product, however the statistical product is not a Frobenius algebra anymore. We summarize this in the following theorem
\begin{theorem}
A statistical manifold $(M,g,\nabla)$ is not flat, if and only if the statistical product $\circ$ is not associative.
\end{theorem}
\begin{proof}
By Proposition \ref{equivelence}, we know the connection $\nabla$ is flat if and only if the potential $\Psi$ (defined locally) does not satisfy the WDVV equation. However the WDVV equation is the associativity condition of a Frobenius algebra.
\end{proof}
\section{Statistical product and Yukawa term}
From the previous section, we know that for a statistical manifold, we can define a product structure in its tangent space
\begin{definition}[Statistical product]
Let $(M,g,\nabla)$ be a statistical manifold, the statistical product $\circ$ on $T_xM$ for $x\in{M}$ is defined by
\begin{equation*}
    \partial_i\circ\partial_j\coloneqq\left(\Gamma^k_{ij}-^{LC}\Gamma^k_{ij}\right)\partial_k,
\end{equation*}
where $\{\partial_i\}$ is a basis of $T_xM$ and $\Gamma^k_{ij}$ and $^{LC}\Gamma_{ij}^k$ are Christoffel symbols of $\nabla$ and the Levi-Civita connection respectively.
\end{definition}
Clearly, the definition does not depend on the choice of the basis and after lowering index, $\Gamma^k_{ij}-^{LC}\Gamma^k_{ij}$ are in fact the components $C_{ijk}$ of the AC tensor up to a scalar factor. From the structural constant $\Gamma^k_{ij}-^{LC}\Gamma^k_{ij}$ of the statistical product or equivalently the AC tensor $C_{ijk}$, we construct the following scalar
\begin{definition}[Yukawa term]
The Yukawa term $C$ is defined as
\begin{equation*}
    C=C_{ijk}C^{ijk}-C_iC^i,
\end{equation*}
where $C_i=C^j_{ij}$ and $C^i=g^{ij}C_j$.
\end{definition}
We believe this scalar quantity contains some information of the statistical manifold just like the scalar curvature constructed from Riemann curvature tensor contains some information of the Riemannian manifold. Since the structural constant $C^i_{jk}$ of the Frobenius algebra play the role of Yukawa coupling in topological field theory, we coin the name Yukawa term.

We calculate two examples from statistical mechanics to show that the Yukawa term contains some useful information. These two examples are inspired by \cite{janyszek}, where they calculated the scalar curvature. The calculation below uses grand canonical ensemble.
\begin{example}[Classical ideal gas]
The free energy per volume is
\begin{equation}
    F = \eta\lambda^{-3},
\end{equation}
where
\begin{equation}
    \lambda=\frac{h}{\sqrt{2\pi mk_BT}}
\end{equation}
is the thermal wavelength, $T$ is the temperature, and $\eta$ is the fugacity. The constants $h$, $m$, $k_B$ appearing in the thermal wavelength are the Planck constant, the mass of a gas molecule, and the Boltzmann constant respectively.

The coordinate system of the statistical manifold is chosen as
\begin{align}
    \begin{split}
        &\beta=\frac{1}{T},\\
        &\gamma=-\ln{\eta}.
    \end{split}
\end{align}
In his coordinate system, the metric tensor induced by the free energy $F$ is
\begin{align}
    \begin{split}
        &g_{\beta\beta}=\frac{15}{4}\lambda^{-3}\beta^{-2}\eta,\\
        &g_{\beta\gamma}=g_{\gamma\beta}=\frac{3}{2}\lambda^{-3}\beta^{-1}\eta,\\
        &g_{\gamma\gamma}=\lambda^{-3}\eta.
    \end{split}
\end{align}
The structural constants of the Frobenius algebra are
\begin{align}
    \begin{split}
        &C_{\beta\beta\beta}=-\frac{105}{8}\lambda^{-3}\beta^{-3}\eta,\\
        &C_{\beta\beta\gamma}=C_{\beta\gamma\beta}=C_{\gamma\beta\beta}=-\frac{15}{4}\lambda^{-3}\beta^{-2}\eta,\\
        &C_{\gamma\gamma\beta}=C_{\gamma\beta\gamma}=C_{\beta\gamma\gamma}=-\frac{3}{2}\lambda^{-3}\beta^{-1}\eta,\\
        &C_{\gamma\gamma\gamma}=-\lambda^{-3}\eta.
    \end{split}
\end{align}
Note that all structural constants are linear in the fugacity $\eta$, which is physically reasonable because they all go to zero as the fugacity goes to zero.

After some calculation, the Yukawa term is
\begin{align}
    \begin{split}
            C&=C_{ijk}C^{ijk}-C_iC^i\\
            &=\frac{20}{3}\lambda^3\eta^{-1}-\frac{20}{3}\lambda^3\eta^{-1}\\
            &=0.
    \end{split}
\end{align}
\end{example}
\begin{remark}
This example shows that the Yukawa term is constantly zero for the classical ideal class, which is in agreement with our physical intuition since the classical ideal gas has no interaction at all.
\end{remark}
Next example is the quantum bosonic ideal gas. Even though the gas still has no interaction, we expect to see qualitatively different behaviour of the Yukawa term caused by quantum fluctuation. We need the polylogarithmic function
\begin{definition}[Polylogarithmic function]
The polylogarithmic function is defined as a Dirichlet series
\begin{equation*}
    \textrm{Li}_s(\eta)=\sum_{k=1}^\infty\frac{\eta^k}{k^s},
\end{equation*}
which converges for all $s\in\mathbb{C}$ if $|\eta|<1$ and can be analytically continued to the whole complex plane.
\end{definition}
From the Dirichlet series representation, it is easy to show the differential relation
\begin{equation}
    \frac{\mathrm{d}\mathrm{Li}_s(\eta)}{\mathrm{d}\eta}=\frac{\mathrm{Li}_{s-1}(\eta)}{\eta}.
\end{equation}
\begin{example}[Quantum bosonic ideal gas]
The free energy per volume is
\begin{equation}
    F = \mathrm{Li}_{\frac{5}{2}}(\eta)\lambda^{-3}.
\end{equation}
The metric tensor $g$ induced by the free energy $F$ is
\begin{align}
    \begin{split}
        &g_{\beta\beta}=\frac{15}{4}\lambda^{-3}\beta^{-2}\mathrm{Li}_{\frac{5}{2}}(\eta),\\
        &g_{\beta\gamma}=g_{\gamma\beta}=\frac{3}{2}\lambda^{-3}\beta^{-1}\mathrm{Li}_{\frac{3}{2}}(\eta),\\
        &g_{\gamma\gamma}=\lambda^{-3}\mathrm{Li}_{\frac{1}{2}}(\eta).
    \end{split}
\end{align}
The structural constants of the Frobenius algebra are
\begin{align}
    \begin{split}
        &C_{\beta\beta\beta}=-\frac{105}{8}\lambda^{-3}\beta^{-3}\mathrm{Li}_{\frac{5}{2}}(\eta),\\
        &C_{\beta\beta\gamma}=C_{\beta\gamma\beta}=C_{\gamma\beta\beta}=-\frac{15}{4}\lambda^{-3}\beta^{-2}\mathrm{Li}_{\frac{3}{2}}(\eta),\\
        &C_{\gamma\gamma\beta}=C_{\gamma\beta\gamma}=C_{\beta\gamma\gamma}=-\frac{3}{2}\lambda^{-3}\beta^{-1}\mathrm{Li}_{\frac{1}{2}}(\eta),\\
        &C_{\gamma\gamma\gamma}=-\lambda^{-3}\mathrm{Li}_{-\frac{1}{2}}(\eta).
    \end{split}
\end{align}
After a tedious calculation, the Yukawa term is
\begin{align}\label{yukawa}
    \begin{split}
        C=&\frac{20\lambda^3}{A^3}\left[5\mathrm{Li}^2_{\frac{5}{2}}(\eta)\mathrm{Li}_{\frac{3}{2}}(\eta)\mathrm{Li}_{\frac{1}{2}}(\eta)\mathrm{Li}_{-\frac{1}{2}}(\eta)\right.\\
        &-10\mathrm{Li}^2_{\frac{5}{2}}(\eta)\mathrm{Li}^3_{\frac{1}{2}}(\eta)-3\mathrm{Li}_{\frac{5}{2}}(\eta)\mathrm{Li}^3_{\frac{3}{2}}(\eta)\mathrm{Li}_{-\frac{1}{2}}(\eta)\\
        &\left.+11\mathrm{Li}_{\frac{5}{2}}(\eta)\mathrm{Li}^2_{\frac{3}{2}}(\eta)\mathrm{Li}^2_{\frac{1}{2}}(\eta)-3\mathrm{Li}^4_{\frac{3}{2}}(\eta)\mathrm{Li}_{\frac{1}{2}}(\eta)\right],
    \end{split}
\end{align}
where
\begin{equation}
    A = 5\mathrm{Li}_{\frac{5}{2}}(\eta)\mathrm{Li}_{\frac{1}{2}}(\eta)-3\mathrm{Li}^2_{\frac{3}{2}}(\eta).
\end{equation}
We consider the asymptotic behaviour of the Yukawa term $C$ as $\eta\rightarrow1^-$ or equivalently as $\gamma\rightarrow0^+$. For $s>1$, the values of polylogarithmic functions at $\eta=1$ coincide with the Riemann-zeta function evaluated at $s$, i.e.
\begin{equation}
    \mathrm{Li}_s(1)=\zeta(s).
\end{equation}
For $s<1$, the polylogarithmic functions has asymptotic property
\begin{equation}
    \mathrm{Li}_s(\eta)\sim\frac{\Gamma(1-s)}{(-\ln\eta)^{1-s}}
\end{equation}
as $\eta\rightarrow1^-$.

Hence
\begin{equation}
    C\sim\frac{2\zeta\left(\frac{3}{2}\right)\lambda^3}{5\sqrt{\pi}\zeta\left(\frac{5}{2}\right)\gamma^\frac{1}{2}}
\end{equation}
as $\gamma\rightarrow0^+$. This singularity of the Yukawa term $C$ can be seen as the onset of the Bose-Einstein condensation.

Next we consider the limit of $C$ as $\eta\rightarrow0^+$. By the series representation of the polylogarithmic function, we know
\begin{equation}
    \mathrm{Li}_s(\eta)\sim\eta
\end{equation}
as $\eta\rightarrow0^+$.
Hence
\begin{equation}
    C\rightarrow0
\end{equation}
as $\eta\rightarrow0^+$. This coincides with the classical ideal gas example.
\end{example}
The Equation (\ref{yukawa}) is complicated. However we show that it is nonnegative
\begin{proposition}
$C\geq0$ for $0\leq\eta\leq1$.
\end{proposition}
\begin{proof}
The Taylor expansion of the bracket in the Equation (\ref{yukawa}) satisfies the inequality
\begin{equation}
    \frac{C(\eta)A^3}{20\lambda^3}>0.33\eta^6+1.29\eta^7+2.85\eta^8+\dots\geq0,
\end{equation}
where all coefficients in the expansion is positive.

Similarly, the Taylor expansion of $A(\eta)$ satisfies the inequality
\begin{equation}
    A(\eta)>2\eta^2+2.2\eta^3+2.3\eta^4+\dots\geq0,
\end{equation}
where all coefficients in the expansion are positive.

It has been shown that $C(0)=0$. The proposition is proved.
\end{proof}
\begin{remark}
This example shows that the Yukawa term for the bosonic ideal gas is nonnegative and the divergence behavious indicates the Bose-Einstein condensation. Furthermore, in classical limit, the Yukawa term goes to $0$, which is in agreement with the classical ideal gas.
\end{remark}
\section{Discussion}
We have shown the intimate relationship between information geometry and Frobenius algebra. First of all, the physical meaning of the AC tensor is clear: It represnets the 3-points function, either in conformal field theory or in topological field theory. From this, we equipped a multiplication structure on the tangent space of a statistical manifold. The commutativity, associativity, and compatibility with metric forces the statistical manifold to have a pencil of flat connections, the so-called Dubrovin connections. We have shown that this pencil of connections entails the dually flat structure in information geometry: The dual connection is simply the contravariant connections associated with the pencil of flat connections. If we do not distinguish covariance and contravariance, the dual connection may be identified within the pencil.

We also have attempted to derive a scalar quantity, which we call Yukawa term, from the AC tensor. We chose this name for two reasons. First, the AC tensor is known as the Yukawa coupling in topological field theory; second, we calculated two examples from statistical mechanics and we found that the Yukawa terms were nonnegative. The nonnegativity qualifies the name Yukawa, since mass should be nonnegative. Interestingly, the Yukawa term for the classical ideal gas is identically zero while the Yukawa term for the quantum bosonic ideal gas indicates the onset of the Bose-Einstein condensation as $\eta\rightarrow0$.

Frobenius algebra is a mathematical concept that unifies different fields, for example integrable system, conformal field theory, Gromov-Witten theory, enumerative geometry etc. Due to the completely different language and background these theories assume, it is difficult to study their full relationship with information geometry here. However, it is certain that some structures and results from these fields have counterparts in information geometry. Inspired by the well-known facts from:
\begin{enumerate}
    \item Tree level conformal field theory that the $n$-points correlation function may be expressed by $3$-points correlation functions;
    \item Gromov-Witten theory, the Reconstruction theorem by Kontsevich and Manin \cite{kontsevich} says that the $n$-points Gromov-Witten invariants can be recursively calculated by the $3$-points Gromov-Witten invariant if the second cohomology group generates the cohomology ring as an algebra;
\end{enumerate}
we ask if there is any reconstruction-type result in information geometry:
\begin{problem}
For what statistical manifold is the (local) potential determined entirely by the metric (2-points function) and the AC tensor (3-points function)?
\end{problem}
Also, the Frobenius algebra may be seen as the $A$- side or the symplectic side of the mirror symmetry. So one may ask 
\begin{problem}
What is the $B$-side of the information geometry?
\end{problem}
Zhang and Khan have done some work toward this question \cite{zhang}.
\bibliographystyle{plain}
\bibliography{references.bib}
\end{document}